\documentclass[10pt,twoside]{amsart}

\usepackage{amsthm}
\usepackage{mathtools}
\usepackage{amssymb}
\usepackage{amsmath}
\usepackage{mathrsfs}
\usepackage{graphicx}
\usepackage{caption}
\usepackage{bm}  
\usepackage{dsfont} 
\usepackage[nice]{nicefrac}   
\usepackage{verbatim}
\usepackage{enumitem}
\usepackage{booktabs}
\usepackage{url}

\newtheorem {Theorem} {Theorem}
\numberwithin{Theorem}{section}

\newtheorem {Proposition}{Proposition}[section]
\theoremstyle{definition}
\newtheorem{Definition}{Definition}[section]
\theoremstyle{remark}
\newtheorem{Remark}{Remark}[section]
\newtheorem{Example}{Example}[section]

\newtheorem {Corollary}{Corollary}

\def    \_      {_{_}}
\def    \C      {\mathds{C}}

\def    \Z      {\mathds{Z}}
\def    \N      {\mathds{N}}
\def    \K      {\mathcal{K}}

\def    \LL     {{\mathfrak L}}
\def    \K      {{\mathfrak K}}

\def    \sst     {\scriptscriptstyle}

\def    \punt   {{\boldsymbol{.}}}
\def    \invumb {{\scriptscriptstyle  \langle-1\rangle}}

\newcommand{\dsum}{\displaystyle\sum}  
\newcommand{\Binom}{\displaystyle\binom} 
\renewcommand{\_}[1]{_{_#1}}

\newcommand{\der}[1]{#1_{_\mathcal{D}}}

\newcommand{\inv}{{^{\sst{-1}}}}
\newcommand{\singleton}{\chi}
\newcommand{\unity}{\upsilon}
\newcommand{\augmentation}{\varepsilon}
\newcommand{\bell}{\beta}
\newcommand{\bernoulli}{\iota}

\newcommand{\equal}{\equiv}

\newcommand{\entries}{\mathfrak{r}}
\newcommand{\entriess}{\mathfrak{s}}

\newcommand{\Rio}{\mathfrak{R}\mathfrak{i}\mathfrak{o}}
\newcommand{\Sheff}{\mathfrak{Sheff}}

\newcommand{\Pascal}{\boldsymbol{P}}

\newcommand{\StirlingI}{\boldsymbol{s}}
\newcommand{\StirlingII}{\boldsymbol{S}}

\allowdisplaybreaks

\title[A symbolic treatment of Riordan arrays]
{A symbolic treatment of Riordan arrays}
\author[J.\ Agapito]{Jos\'e Agapito}
\author[\^A.\ Mestre]{\^Angela Mestre}
\author[P.\ Petrullo]{Pasquale Petrullo}
\author[M.M.\ Torres]{Maria M. Torres}
\address{Centro de Estruturas Lineares e Combinat\'orias, Universidade de Lisboa, Portugal}
\email{jagruiz@cii.fc.ul.pt}
\email{mestre@cii.fc.ul.pt}
\email{mmtorres@fc.ul.pt}
\address{Universit\'a degli studi della Basilicata, Potenza, Italy}
\email{p.petrullo@gmail.com}
\thanks{This work was done within the activities of the Centro de Estruturas Lineares e Combinat\'orias (University of Lisbon, Portugal) and the Dipartimento di Matematica e Informatica (Universit\`a degli Studi della Basilicata, Italy). The first author was partially supported by the Portuguese Science and Technology Foundation (FCT) through the program Ci\^encia 2008 and grant PTDC/MAT/099880/2008; the second author was supported by the FCT grant SFRH/BPD/ 48223/2008 and the fourth author was partially supported by the FCT and FEDER/POCI 2010.}

\begin{document}

\begin{abstract}
We approach Riordan arrays and their generalizations via umbral symbolic methods. This new approach allows us to derive fundamental aspects of the theory of Riordan arrays as immediate consequences of the umbral version of the classical Abel's identity for polynomials. In particular, we obtain a novel non-recursive formula for Riordan arrays and derive, from this new formula, some known recurrences and a new recurrence relation for Riordan arrays.
\vspace{-1cm}
\end{abstract}

\maketitle


\section{Introduction}
\label{se:intro}

A Riordan array is an infinite lower triangular matrix $R=(r_{n,k})_{n,k\in\N}$ \emph{ordinarily}  described by a pair of generating functions $(d(t),h(t))$ such that $d(0)\neq 0$, $h(0)=0$,  $h'(0)\neq 0$, and $r_{n,k}=\left[t^n\right]d(t)\,h(t)^k$, where $[t^n]$ is the operator which gives the $n^{\rm{th}}$ coefficient in the series development of a generating function. If $d(t)=\sum_{n\geq 0}d_n\frac{t^n}{n!}$, $h(t)=\sum_{n\geq 1}h_n\frac{t^n}{n!}$, and $r_{n,k}=n!\left[t^n\right]d(t)\frac{h(t)^k}{k!}$, then $R$ is called \emph{exponential} Riordan array. More generally, given a sequence $(c_n)$ of nonzero numbers and taking $d(t)=\sum_{n\geq 0}\,d_n\frac{t^n}{c_n}$, $h(t)=\sum_{n\geq 1}\,h_n\frac{t^n}{c_n}$, and $r_{n,k}=c_n\left[t^n\right]d(t)\frac{h(t)^k}{c_k}$, $R$ is called a \emph{generalized} Riordan array with respect to $(c_n)$ \cite{WW}. Thus, for example, the Riordan array whose entries are the binomial numbers $\binom{n}{k}$ can be described by the pair $\big(\frac{1}{1-t},\frac{t}{1-t}\big)$ ($c_n=1$: ordinary presentation) or by $(e^t,t)$ ($c_n=n!$: exponential presentation). Riordan arrays form a group under matrix multiplication. The literature about Riordan arrays is vast and still growing and the applications cover a wide range of subjects, such as enumerative combinatorics, combinatorial sums, recurrence relations and computer science, among other topics \cite{Chen,LMMS:identities,LM:recurrence,M08,MRSV,MS10,MS11,sha:bijection,ShGeWoWo:Rgroup,Sprugnoli:Rioarray-combsum,Sprugnoli:bib}. Formally, ordinary Riordan arrays are a formulation of the $1$-umbral calculus, whereas exponential Riordan arrays are a formulation of the $n!$-umbral calculus. In fact, the classical umbral calculus, as given in \cite{Roman:umbral}, consists of a systematic study of a certain class of univariate polynomial families (\emph{Sheffer sequences}) by employing linear operators on polynomials. An extension of the classical umbral calculus to infinitely many variables can be found in \cite{BBN83}. Sheffer sequences, provided with a particular non-commutative multiplication (\emph{umbral composition}), form a group which is isomorphic to the Riordan group \cite{HHS}.

The purpose of this paper is to give a promising new symbolic treatment of Riordan arrays, based on a renewed approach to umbral calculus initiated by Rota and Taylor in  \cite{RT:classicalumbral}. This renewed approach makes no use of operator theory. The symbolic techniques developed from this new approach have been fruitfully applied to a wide range of topics \cite{diNNS:Sheffer,diNPS:cumulants-convolutions,diNS:poisson,diNS:cumulant-factorial,Ges,diNGS:bernoulli,Petrullo:abelpolynomials,Petrullo:thesis,Petrullo:abelidentity,
PetSen:puma,RS,RST:abelpolynomials,stanley:eul,Tay}. In particular, Di Nardo, Niederhausen and Senato \cite{diNNS:Sheffer} gave a representation of exponential Riordan arrays as a result of their symbolic handling of Sheffer polynomials. Definition \ref{de:riordanarray} given later on this paper is a normalized version of their formulation. A treatment of Riordan arrays under the more general context of \emph{recursive} matrices using the more established approach to umbral calculus \cite{Roman:umbral} can be found in \cite{BBN82,BM}.

A key tool in our symbolic approach to Riordan arrays is the umbral version (formula \eqref{eq:abel-identity}) of the following binomial Abel identity for polynomials 
\begin{equation}\label{eq:Abelidentity}
(x+y)^n=\dsum_{k=0}^n\binom{n}{k}(y+ka)^{n-k} x(x-ka)^{k-1}\,.
\end{equation}
The umbral analogue of \eqref{eq:Abelidentity} was first stated in \cite{diNS:poisson} (without a proof) and later obtained from an umbral version of the Lagrange inversion formula in \cite{diNNS:Sheffer}. An elementary proof of the umbral version of Abel's identity has been recently given by Petrullo \cite{Petrullo:thesis,Petrullo:abelidentity}.

The main contribution of this paper is to show that well-known properties of all Riordan arrays are easily derived from the umbral version of Abel's identity. This is a nontrivial aspect from a theoretical point of view. It is not immediate by means of the traditional methods for dealing with Riordan arrays. It is worth mentioning that, in the classical context of generating functions,  Sprugnoli \cite{Sprugnoli:Riordanarrays-AbelGould} had already used Riordan arrays to obtain several combinatorial formulas that generalize the classical Abel's identity. Therefore, our point of view is conceptually different from Sprugnoli's approach. 

The paper is organized as follows. Section \ref{se:prelim} recalls the basics of Rota and Taylor's classical umbral calculus \cite{RT:classicalumbral}. Along the way, two equivalent umbral versions of Abel's identity for polynomials and an umbral version of the classical Lagrange inversion formula are given. Section \ref{se:riordan} states the fundamental aspects of the theory of Riordan arrays (of exponential type) in umbral terms. Also, connections with Sheffer sequences are recast in the new umbral symbolic  setting. Furthermore, some recurrence relations of Riordan arrays are stated. Section \ref{se:examples} tests our umbral approach to Riordan arrays on some concrete classical examples, providing both umbral and traditional formulas. Section \ref{se:genrio} extends the discussion given in Section \ref{se:riordan} to $\omega$-Riordan arrays, where $\omega$ is any umbra with nonzero moments. In this generalized context, we obtain an important umbral formula (Theorem \ref{th:mother-rec}) as  a direct consequence of the umbral Abel identity \eqref{eq:abel-identity}.  One of our main results is a novel non-recursive formula (Theorem \ref{th:nonrec}), which is a direct application of Theorem \ref{th:mother-rec}. Known recurrence relations and a new recurrence formula for $\omega$-Riordan arrays (Theorems \ref{th:rechor}, \ref{th:recver} and \ref{th:recdiff}) are derived from Theorem \ref{th:nonrec}.
\section{Preliminaries}
\label{se:prelim}
This section starts with a brief review of the essentials of the classical umbral calculus that underlies much of this paper. We recall the intimate relation that exists between the umbral version of Abel's identity \eqref{eq:Abelidentity} and the Lagrange inversion formula. This connection is not clear at first sight from the traditional point of view using generating functions (see for instance \cite{Comtet} for the common view of these two classical formulas). The umbral versions of the binomial Abel identity and of the Lagrange inversion formula play an important role in our symbolic approach to Riordan arrays (see Sections~\ref{se:riordan} and ~\ref{se:genrio}).

\subsection{The basics of classical umbral calculus} 

To perform classical umbral calculus we need to have:
\begin{enumerate}[itemsep=1ex,leftmargin=0.8cm]
\item[(i)] A commutative integral domain $R$ with identity $1$ (assume  $R=\C[x]$, $x$ possibly multivariate).
\item[(ii)] An \emph{alphabet} $A=\{\alpha,\beta,\gamma,\ldots\}$ with \emph{enough} letters called \emph{umbrae}.
\item[(iii)] An $R$-linear functional $E\colon R[A]\to R$ such that $E[1]=1$ and $E[\alpha^i\bell^j\cdots\gamma^k]=E[\alpha^i]E[\bell^j]\cdots E[\gamma^k]$, for pairwise distinct umbrae $\alpha,\bell,\ldots,\gamma$  (\emph{uncorrelation} property). The map $E$ is called \emph{evaluation}.
\end{enumerate}

\noindent $R[A]$ simply denotes the polynomial ring in the set $A$ of commuting indeterminates, with coefficients in $R$, endowed with formal addition and multiplication. An element $p\in R[A]$ is called \emph{umbral polynomial}. The \emph{support} of an umbral polynomial $p\in R[A]$ is defined as the set of all umbrae that show up as variables of $p$. Two umbral polynomials $p,q\in R[A]$ are said to be \emph{uncorrelated} if their supports are disjoint. In particular, distinct umbrae are always uncorrelated. Ordinary equality defines a trivial equivalence relation on $R[A]$. Additionally,  two nontrivial equivalence relations can be defined on $R[A]$. We write $p\simeq q$ and say that $p$ and $q$ are umbrally \emph{equivalent}, when $E[p]=E[q]$. Moreover, if $p^n\simeq q^n$ for all integers $n\ge 0$, we say that $p$ and $q$ are \emph{similar} and write $p\equiv q$. An umbral polynomial $p$ is said to represent a sequence $(a_n)$ in $R$ (with $a_0=1$) if $E[p^n] = a_n$ for all $n\ge 0$. In this case we say that $a_n$ is the $n^{\rm{th}}$ \emph{moment} of $p$ and write $p^n\simeq a_n$. The equivalent relation $p\equal q$ means that $p$ and $q$ represent the same sequence. Often, it is convenient to assume that each monic sequence $(a_n)$ in $R$ is represented by infinitely many distinct similar umbrae. In this case we deal with a \emph{saturated umbral calculus} or, equivalently, with a saturated alphabet \cite{RT:classicalumbral}. We can then rephrase property (ii) in a more precise manner by saying that $A$ is saturated.
\begin{Remark}\label{re:3e-r} Although two equal polynomials are similar and two similar polynomials are umbrally equivalent, neither converse holds. For instance, consider umbrae $\alpha, \alpha', \gamma, \gamma'$ such that $\alpha\neq\alpha'$, $\alpha\equal\alpha'$ and $\gamma\neq\gamma'$, $\gamma\equiv\gamma'$. Clearly $\alpha+\gamma\neq \alpha'+\gamma'$, but  $\alpha+\gamma\simeq\alpha'+\gamma'$. Moreover, we have $\alpha+\gamma\equal\alpha'+\gamma'$. Similarly, note that $\alpha+\alpha'\simeq 2\alpha$, whereas $\alpha+\alpha'\not\equal 2\alpha$. These examples show that identities in $R[A]$ involving umbral equivalence and similarity encode more information than ordinary equality.
\end{Remark}
Given $p\in R[A]$, we denote by $e^{p t} $ the formal exponential series
\begin{equation*}\label{eq:gf-umbra}
e^{p t}:=1+\dsum_{n=1}^\infty
p^n\,\frac{t^n}{n!}\,\,\in\,\,R[A]\,[[t]]\,.
\end{equation*}
The series above is called the \emph{generating function} (g.f. for short) of $p$. We extend the action of $E$ from $R[A]$ to $R[A][[t]]$ and write $f_p(t)=E[e^{p t}]:=1+\sum_{n=1}^\infty E[p^n]\,\frac{t^n}{n!}\in R[[t]]$. The series $f_p(t)$ is called the \emph{moment generating function} (m.g.f. for short) of $p$. As done for umbral polynomials, we may also define umbral equivalence on $R[A][[t]]$ and write $e^{p t}\simeq e^{q t}$ if and only if $E[e^{p t}]\simeq E[e^{q t}]$, so that $e^{p t}\simeq f_p(t)$. It is not difficult to see that $f_p(t)=f_q(t)$ if and only if $p\equiv q$. We will often use the terms m.g.f. and g.f without distinction throughout the paper. For later reference, we summarize in Table~\ref{tb:keyumbrae} useful information on some fundamental  umbrae.\par

\smallskip

\begin{table}[h]
  \centering
  \begin{tabular}{lccccc}
               \toprule
               \textbf{Umbra} &  $\bm{\alpha}$ & & $\bm{f_\alpha(t)=E[e^{\alpha t}]}$ & & $\bm{E[\alpha^n]=n![t^n]f_\alpha(t)}$ \\
               \midrule
               Augmentation & $\augmentation$ & & $1$ & & $1,0,0,\ldots$ \\[1ex]
               Unity & $\unity$ & & $e^{\,t}$ & & $1,1,1,\ldots$ \\[1ex]
               Singleton & $\singleton$ & & $1+\,t$ & & $1,1,0,\ldots$ \\[1ex]
               Bell & $\bell$ & & $e^{e^{\,t}-1}$ & & $1,1,B_2,B_3,\ldots$ (Bell numbers) \\[1ex]
               Bernoulli & $\iota$ & & $\dfrac{t}{e^{\,t}-1}$ & & $1,b_1,b_2,\ldots$ (Bernoulli numbers) \\[1ex]
               Boolean unity & $\bar{\unity}$ & & $\frac{1}{1-t}$ & & $0!,1!,2!,3!,\ldots$ \\
               \bottomrule
\end{tabular}
  \caption{Some fundamental umbrae.}\label{tb:keyumbrae}
\vspace{-0.5cm}
\end{table}

\noindent On the one hand, note that $f_{\alpha+\gamma}(t)=E[e^{(\alpha+\gamma)t}]=E[e^{\alpha t} e^{\gamma t}]=E[e^{\alpha t}]E[e^{\gamma t}]=f_\alpha(t)\,f_\gamma(t)$ if $\alpha\neq\gamma$ (uncorrelation property). On the other hand, clearly $f_{2\alpha}(t)=E[e^{(2\alpha)t}]=E[e^{\alpha(2t)}]=f_\alpha(2t)\neq (f_\alpha(t))^2$. However, if $\alpha\neq\alpha'$ and $\alpha\equiv\alpha'$, then  $f_{\alpha+\alpha'}(t)=(f_\alpha(t))^2$. The inequality $f_{\alpha+\alpha'}(t)\neq f_{2\alpha}(t)$ reflects the fact that the umbral polynomials $\alpha+\alpha'$ and $2\alpha$ are not similar, as was pointed out in Remark~\ref{re:3e-r}. We stress this difference and set  $2\punt\alpha\equal\alpha+\alpha'$ (with $\alpha\neq\alpha'$ and $\alpha\equal\alpha'$). Thus, for all $n\ge 0$, we have
\begin{equation*}
(2\punt\alpha)^n\simeq(\alpha+\alpha')^n\simeq\dsum_{k=0}^n\binom{n}{k}\alpha^k\,
(\alpha')^{n-k}\simeq\dsum_{k=0}^n\binom{n}{k}a_k\,a_{n-k}.
\end{equation*}
\noindent Up to similarity, the auxiliary umbra $2\punt\alpha$ is well defined. In general, given $x\in R\cup A$, we shall implicitly define the auxiliary umbra $x\punt\alpha$ by means of its g.f.; $e^{(x\punt\alpha)t}\simeq (f_\alpha(t))^x $; that is, $(x\punt\alpha)^n\simeq n![t^n](f_\alpha(t))^x$. In particular, for any $\gamma, \alpha\in A$, we have $e^{(\gamma\punt\alpha)t}\simeq (f_\alpha(t))^\gamma\simeq e^{\gamma(\log f_\alpha(t))}$; namely, $f_{\gamma\punt\alpha}(t)=f_\gamma\big(\log f_\alpha(t)\big)$. 

\begin{Remark}\label{re:momentsofdotproduct}
Let $\alpha^n\simeq a_n$ for all $n\in\N$. The moments of $\gamma\punt\alpha$ are explicitly given by the formula
\begin{equation*}
(\gamma\punt\alpha)^n\simeq\dsum_{i=0}^n (\gamma)_i B_{n,i}(a_1,a_2,\ldots,a_{n-i+1})\,,
\end{equation*}
where $(\gamma)_i=\gamma(\gamma-1)\cdots(\gamma-i+1)$ and the $B_{n,i}(a_1,a_2,\ldots,a_{n-i+1})$ are the partial Bell exponential polynomials (see \cite{diNS:poisson} for details). Note that $\gamma\punt\alpha\simeq\gamma\,\alpha$ but $\gamma\punt\alpha\not\equal\gamma\,\alpha$.
\end{Remark}

The dot product of umbrae is in general non-commutative. Nevertheless, it satisfies several useful identities which can be easily verified via generating functions.

\begin{Example}
Given $\alpha\in A$, $r\in R$ and the umbrae $\augmentation$, $\bell$, $\unity$ and $\singleton$ whose g.f.'s are shown in Table~\ref{tb:keyumbrae}, we have
\begin{equation*}
\alpha\punt\augmentation\equiv\augmentation\equiv\augmentation\punt\alpha\,\,,\,\,\singleton\punt\bell\equiv\bell\punt\singleton\equiv\unity\,\,,\,\, r\alpha\equiv\alpha\punt(r\unity)\,\,,\,\,
\alpha+\augmentation\equal\alpha\equal\augmentation+\alpha\,\,,\,\, \alpha+(-1)\punt\alpha\equal\augmentation\equal(-1)\punt\alpha+\alpha\,.
\end{equation*}
For instance, $f_{\augmentation\punt\alpha}(t)=f_\augmentation(\log f_\alpha(t))=1$ and $f_{\alpha\punt\augmentation}(t)=f_\alpha(\log f_\augmentation(t))=f_\alpha(0)=1$; hence $\alpha\punt\augmentation\equiv\augmentation\equiv\augmentation\punt\alpha$. Similarly, $f_{\singleton\punt\bell}(t)=f_\singleton(\log f_\bell(t))=f_\singleton(\log e^{e^t-1})=1+e^t-1=e^t$ and $f_{\bell\punt\singleton}(t)=f_\bell(\log f_\singleton(t))=f_\bell(\log(1+t))=e^{e^{\log(1+t)}-1}=e^{1+t-1}=e^t$; hence $\singleton\punt\bell\equiv\bell\punt\singleton\equiv\unity$.
\end{Example}

By iterating the dot product we can define further auxiliary umbrae. In this way, the umbra $\gamma\punt\bell\punt\alpha$ (where $\bell$ is the Bell umbra) is called the \emph{composition} umbra of $\gamma$ with $\alpha$ because it encodes the composition of formal series; namely, $ f_{\gamma\punt\bell\punt\alpha}(t)=f_\gamma\big(f_\alpha(t)-1\big)$.
We can also define new umbrae out of a single umbra. For example, $\alpha^\invumb$ denotes the  auxiliary umbra that satisfies $\alpha^\invumb\punt\bell\punt\alpha\equiv\singleton\equiv\alpha\punt\bell\punt\alpha^\invumb$. We call $\alpha^\invumb$ the compositional inverse of $\alpha$. One can check that $f_{\alpha^\invumb}(t)-1=(f_\alpha(t)-1)^\invumb$. Likewise, we define the $\alpha$-\textit{derivative} umbra $\der{\alpha}$ as the auxiliary umbra whose moments are given by $\der{\alpha}^n\simeq n\alpha^{n-1}$ for all $n\ge 1$. It is easy to see that the m.g.f. of $\der{\alpha}$ satisfies $f_{\der{\alpha}}(t)=1+t\,f_\alpha(t)$. Since $(t\,f_{\alpha}(t))^m\simeq t^m\,e^{(m\punt\alpha)t}$ for any integer $m\ge 0$, we have 
$
e^{(\sigma\punt\bell\punt\der{\alpha})t}\simeq e^{\sigma[tf_\alpha(t)]}\simeq \sum_{m=0}^\infty\sum_{k=0}^\infty\sigma^m\,(m\punt\alpha)^k\,\frac{t^m}{m!}\frac{t^k}{k!}
$. 
Comparing the coefficients of $\frac{t^n}{n!}$, we obtain the following binomial-like expansion for the moments of a composition umbra,
\begin{equation}\label{eq:comp-mom}
(\sigma\punt\bell\punt\der{\alpha})^n\simeq\sum_{k=0}^n\binom{n}{k}\sigma^k(k\punt\alpha)^{n-k}\,.
\end{equation}
\begin{Remark} Let $f_\sigma(t)$ be the m.g.f. of $\sigma$. Formula \eqref{eq:comp-mom} is equivalent to
\begin{equation}\label{eq:comp-mom-classic}
[t^n]f_\sigma(tf_\alpha(t)) = \sum_{k=0}^n [t^k] f_\sigma(t) \, [t^{n-k}] (f_\alpha(t))^k \,. 
\end{equation}
\end{Remark}

\subsection{Umbral Abel identity and the Lagrange inversion formula}
\label{sse:umbral-abelpoly}

Let $\alpha, \gamma$ and $\sigma$ be any umbrae. For all $n\ge 1$, we have
\begin{equation}\label{eq:abel-identity}
(\gamma+\sigma)^n\simeq\dsum_{k=0}^n\binom{n}{k}(\gamma+k\punt\alpha)^{n-k} \sigma\big(\sigma+(-k)\punt\alpha\big)^{k-1}\,\,\,\,\qquad\text{(Umbral Abel Identity, Version I)}.
\end{equation}
Formula~\eqref{eq:abel-identity} is a straightforward (though nontrivial) generalization of the classical Abel identity for polynomials in three indeterminates shown in \eqref{eq:Abelidentity}. An elementary proof and further applications of \eqref{eq:abel-identity} have been recently given by Petrullo \cite{Petrullo:abelidentity}. The umbral Abel polynomials $\sigma\big(\sigma+(-n)\punt\alpha\big)^{n-1}$ are direct generalizations of the classical Abel polynomials $x(x-na)^{n-1}$ for $a\in R$. Such polynomials were first studied by Rota, Shen and Taylor \cite{RST:abelpolynomials} in connection with polynomials of binomial type. For the sake of clarity, our notation for umbral Abel polynomials slightly differs from the usual convention: $\sigma\big(\sigma-n\punt\alpha\big)^{n-1}$. Umbral Abel polynomials have been applied to cumulants and convolutions in probability theory by Di Nardo, Petrullo and Senato~\cite{diNPS:cumulants-convolutions}, and further generalized by Petrullo in \cite{Petrullo:abelpolynomials,Petrullo:thesis}. For convenience, let us now denote by $\K_{\sigma,\alpha}$ an auxiliary umbra whose moments are given by
\begin{equation}\label{eq:K}
\K_{\sigma,\alpha}^n\simeq \sigma\big(\sigma+(-n)\punt\alpha\big)^{n-1}\quad,\quad n\geq 1.
\end{equation}
A simple proof of the Lagrange inversion formula immediately follows.
\begin{Theorem}[Lagrange inversion formula] \label{th:lagrange} For any umbrae $\alpha, \gamma$ and all integers $n\ge 1$, we have
\begin{equation}\label{eq:lagrange}
\sigma\big(\sigma+(-n)\punt\alpha\big)^{n-1}\simeq(\sigma\punt\beta\punt\der{\alpha}^\invumb)^n\,.
\end{equation}
\end{Theorem}
\begin{proof} By applying \eqref{eq:comp-mom}, then \eqref{eq:K}, and finally \eqref{eq:abel-identity}, we get
\begin{eqnarray*}(\K_{\sst\sigma,\alpha}\punt\beta\punt\der{\alpha})^n  \simeq\sum_{k=0}^n\binom{n}{k}\sigma\big(\sigma+(-k)\punt\alpha\big)^{k-1}(k\punt\alpha)^{n-k}\simeq
(\augmentation+\sigma)^n\simeq\sigma^n.
\end{eqnarray*}
Therefore, $\K_{\sst\sigma,\alpha}\punt\beta\punt\der{\alpha}\equiv \sigma$, which implies  that $\K_{\sst\sigma,\alpha}\equiv\sigma\punt\beta\punt\der{\alpha}^\invumb$.
\end{proof}
A plain translation of \eqref{eq:lagrange} in terms of m.g.f.'s of umbrae gives the traditional Lagrange inversion formula
\begin{equation}\label{eq:typicalLIF}
\big[t^{n-1}\big]f'_\sigma(t)\left(\frac{1}{f_\alpha(t)}\right)^n=n\big[t^n\big]f_\sigma\big((tf_\alpha(t))^\invumb\big)\,.
\end{equation}
Now, let $\eta\equiv\K_{\sigma,\alpha}$. Since $\eta\punt\beta\punt\der{\alpha}\equiv\sigma\punt\bell\punt\der{\alpha}^\invumb\punt\bell\punt\der{\alpha}\equiv\sigma$, we can rewrite \eqref{eq:abel-identity} as
\begin{equation}\label{eq:FTRAnk}
(\gamma+\eta\punt\bell\punt\der{\alpha})^n\simeq\sum_{k=0}^n\binom{n}{k}(\gamma+k\punt\alpha)^{n-k}\eta^k\,\,,\,\, n\ge 0.\qquad\text{(Umbral Abel Identity, Version II)}
\end{equation}
\noindent We shall see in Section~\ref{se:riordan} that formula \eqref{eq:FTRAnk} gives an umbral version of the fundamental theorem of Riordan arrays. In addition, it follows from \eqref{eq:comp-mom} and \eqref{eq:FTRAnk} the next useful umbral identity,
\begin{equation}\label{eq:der}\der{\left(\alpha+\eta\punt\beta\punt
\der{\alpha}\right)}\equiv\der{\eta}\punt\beta\punt\der{\alpha}.\end{equation}
Let us introduce our last umbral tool. By $\LL_{\sst\gamma,\alpha}$, we denote an auxiliary umbra that satisfies  $\LL_{\sst\gamma,\alpha}\equiv-1\punt\K_{\sst\gamma,\alpha}$. We write $\K_{\sst\alpha}:=\K_{\sst\alpha,\alpha}$ and $\LL_{\sst\alpha}:=\LL_{\sst\alpha,\alpha}$. The umbrae $\K_{\sst\alpha}$ and $\LL_{\sst\alpha}$ play a central role in the umbral theory of cumulants~\cite{diNPS:cumulants-convolutions,Petrullo:thesis}. Since $\LL_{\sst\alpha}\equiv-1\punt\alpha\punt\beta\punt\der{\alpha}^\invumb$ and $\der{\varepsilon}\equiv\singleton$, by \eqref{eq:der} we have $\chi\equiv\der{\left(\alpha+\LL_{\sst\alpha}\punt\beta\punt\der{\alpha}\right)}\equiv\der{(\LL_{\sst\alpha})}\punt\beta\punt\der{\alpha}$, so that
\begin{equation}\label{eq:LL}\der{(\LL_{\sst\alpha})}\equiv\der{\alpha}^\invumb.\end{equation}
In particular, we have $\LL_{\LL_{\alpha}}\equiv\alpha$. It is for this reason that $\LL_{\sst\alpha}$ is named the \emph{Lagrange involution} of $\alpha$. Now, for all $k\geq 1$ denote by $\delta^{\sst(k)}$ an umbra whose g.f. $f_{\sst \delta^{\sst(k)}}(t)=1+\frac{t^k}{k!}$. By \eqref{eq:LL}, we have $\delta^{\sst(k)}\punt\beta\punt\der{(\LL_{\sst\alpha})}\equiv\delta^{\sst(k)}\punt\beta\punt\der{\alpha}^\invumb$. Hence, by \eqref{eq:comp-mom} and \eqref{eq:lagrange}, we get $\binom{n}{k}(k\punt\LL_{\sst\alpha})^{n-k}\simeq\binom{n-1}{k-1}(-n\punt\alpha)^{n-k}$, so that
\begin{equation}\label{eq:lagrangek1}n(k\punt\LL_{\sst\alpha})^{n-k}\simeq k(-n\punt\alpha)^{n-k}.\end{equation}
Since $\LL_{\LL_{\alpha}}\equiv\alpha$ and $-n\punt\LL_{\sst\alpha}\equal n\punt\K_{\sst\alpha}$, formula \eqref{eq:lagrangek1} is equivalent to 
\begin{equation}\label{eq:lagrangek2}n(k\punt\alpha)^{n-k}\simeq k(n\punt\K_{\sst\alpha})^{n-k}.\end{equation} 
\begin{Remark}
\noindent As pointed out by G. Andrews \cite{And}, at first glance, the classical umbral calculus seems to be merely a convenient notation for dealing with exponential generating functions. However, it reveals its power in tackling difficult problems concerning generating functions not easily handled by usual algebraic techniques \cite{Ges,DNGS09}. The aforementioned  umbral identities (involving dot products, umbral equivalence and similarity) show not only more appealing formulas but also, from our perspective, more effective ones  since they  allow for further algebraic manipulations that help us to write in polynomial terms identities otherwise stated in terms of the operator $[t^n]$ and generating functions (compare e.g., \eqref{eq:comp-mom} with \eqref{eq:comp-mom-classic} and \eqref{eq:lagrange} with \eqref{eq:typicalLIF}). 
\end{Remark}
\section{Exponential Riordan arrays}
\label{se:riordan}

In this section we use the umbral syntax developed in Section~\ref{se:prelim} to encode Riordan arrays. We show that all the basic facts about Riordan arrays are easily stated and derived by means of elucidating umbral identities. For simplicity, we test the effectiveness of the umbral notation on normalized Riordan arrays; namely, those having diagonal entries all equal to $1$. We provide a simple umbral characterization of the group axioms, describe some important Riordan subgroups, and recover the isomorphism between the Riordan and the Sheffer groups. Likewise, recursive properties of Riordan arrays are obtained as direct  consequences of the umbral Abel identity. These properties will be further generalized in Section \ref{se:genrio}.

\subsection{The exponential Riordan group}
\label{sse:riordanarray}

Riordan arrays are a generalization of the classical Pascal array, whose entries are the well-known binomial coefficients $\binom{n}{k}$ \cite{ShGeWoWo:Rgroup}. This feature, which is not put in evidence in the traditional description of a Riordan array (see Section~\ref{se:intro}), is highlighted by the umbral syntax in the next definition.
\begin{Definition}\label{de:riordanarray}
Given any umbrae $\alpha$ and $\gamma$, the pair $(\gamma,\alpha)$ denotes an infinite lower triangular matrix whose entries $(\gamma,\alpha)_{n,k}$ are given by
\begin{equation}\label{eq:nk-entry-umbral}
(\gamma,\alpha)_{n,k}\simeq \binom{n}{k}\,(\gamma+k \punt\alpha)^{n-k}\quad,\quad\text{for}\quad n\ge k\ge 0\,.
\end{equation}
\noindent We say that $(\gamma,\alpha)$ is an exponential Riordan array generated by  $\alpha$ and $\gamma$. Exponential Riordan arrays of type $(p,q)$, for $p,q\in R[A]$, are defined analogously. Note that if $e^{\gamma t}\simeq d(t)$ and $te^{\alpha t}\simeq h(t)$, with $d(0)=1$, $h(0)=0$, and $h'(0)=1$, then we get a more familiar presentation for \eqref{eq:nk-entry-umbral}; namely
\begin{equation*}
r_{n,k} = (\gamma,\alpha)_{n,k} = E\left[\binom{n}{k}\,(\gamma+k \punt\alpha)^{n-k}\right] = n!\left[t^n\right] d(t)\,\frac{h(t)^k}{k!}\,.
\end{equation*}
\end{Definition}

\noindent In the umbral notation, the identity array is represented by $(\augmentation,\augmentation)$ and the Pascal array by $(\unity,\augmentation)$, where $\unity$ is the unity umbra and $\augmentation$ is the augmentation umbra. Note that, given any umbrae $\gamma$ and $\alpha$, we have $(\gamma,\alpha)_{n,n}= 1$ for all $n\ge 0$. Because of this property, we say that $(\gamma,\alpha)$ is normalized. Any Riordan array $(d(t),h(t))$ with $d(0)\neq 0$, $h(0)=0$, and $h'(0)\neq 0$ can be normalized in a natural way; hence restricting our attention to normalized Riordan arrays is not a loss of generality. From now on, we will simply say Riordan array to mean normalized exponential Riordan array.

Let $f(t)$ and $g(t)$ be the exponential g.f.'s of the sequences $(f_i)$ and $(g_i)$, with $f_0=g_0=1$. In classical terms, a Riordan array $(d(t),h(t))$ relates $f(t)$ with $g(t)$ by means of $(d(t),h(t))f(t)=d(t)f(h(t))=g(t)$. This identity is known as the fundamental theorem in the theory of Riordan arrays. In umbral terms, we shall write $(\gamma,\alpha)\eta\equiv\omega$, for  $\eta$ and $\omega$ in $A$ such that $f_\eta(t)=f(t)$ and $f_\omega(t)=g(t)$. The next result is a direct consequence of identity \eqref{eq:FTRAnk} and Definition \ref{de:riordanarray}.

\begin{Theorem}[Fundamental Theorem of Riordan arrays]\label{th:FTRA}
For all umbrae $\alpha,\gamma,\eta$ we have
\begin{equation}\label{eq:FTRA}
(\gamma,\alpha)\eta\equal\gamma+\eta\punt\bell\punt\der{\alpha}.
\end{equation}
\end{Theorem}

The fundamental theorem of Riordan arrays is nothing but a reformulation of the umbral Abel identity. Observe also that 
\begin{equation*}
\big((\gamma,\alpha)\eta\big)^n\simeq\dsum_{k=0}^n(\gamma,\alpha)_{n,k}\,\,\eta^k\simeq\dsum_{k=0}^n r_{n,k}f_k\,.
\end{equation*}
\noindent Hence, after applying the evaluation map $E$, formula \eqref{eq:FTRA} gives the summation formula 
\begin{equation*}
\sum_{k=0}^n r_{n,k}\,f_k = n! \left[t^n\right] d(t) f(h(t))\,.
\end{equation*}
\noindent By taking $\eta\equiv\unity$ in \eqref{eq:FTRA}, the resulting umbral polynomial $\gamma+\beta\punt\der{\alpha}$ represents the sequence of numbers whose $n^{\rm{th}}$ term is the sum of all entries occurring in the $n^{\rm{th}}$ row of $(\gamma,\alpha)$; namely,
\begin{equation*}
\big(\gamma+\beta\punt\der{\alpha}\big)^n\simeq \sum_{k=0}^n(\gamma,\alpha)_{n,k}\,.
\end{equation*}
\noindent With the same ease, the product of Riordan arrays is encoded in umbral terms as follows.
\begin{Proposition}\label{th:Riomul} Let $(\gamma,\alpha)$ and $(\sigma,\rho)$ be Riordan arrays. We have 
 \begin{equation}\label{eq:Riomul}
(\gamma,\alpha)(\sigma,\rho)=(\gamma+\sigma\punt\bell\punt\der{\alpha}\,,\,\alpha+\rho\punt\bell\punt\der{\alpha})\,.
\end{equation}
\end{Proposition}
\begin{proof}
Let $\eta$ be any umbra. From the fundamental theorem and identity \eqref{eq:der}, we have
 \begin{eqnarray*} (\gamma,\alpha)(\sigma,\rho)\eta\equiv
(\gamma,\alpha)(\sigma+\eta\punt\bell\punt\der{\rho})\equiv (\gamma+\sigma\punt\bell\punt\der{\alpha})+
\eta\punt\bell\punt \alpha+\der{\rho}\punt\bell\punt\der{\alpha}\\
\equiv
(\gamma+\sigma\punt\bell\punt\der{\alpha})+
\eta\punt\bell\punt\der{(\alpha+\rho\punt\bell\punt\der{\alpha})}\equiv (\gamma+\sigma\punt\bell\punt\der{\alpha}\,,\,\alpha+\rho\punt\bell\punt\der{\alpha})\eta
\end{eqnarray*}
Formula \eqref{eq:Riomul} follows from here.
\end{proof}
\noindent Setting $f_\gamma(t)=d(t)$, $tf_\alpha(t)=h(t)$, $f_\sigma(t)=a(t)$, and $tf_\rho(t)=b(t)$, we recover the traditional formula for the product of the  Riordan arrays $(d(t),h(t))$ and $(a(t),b(t))$; that is, 
$$
\big(d(t), h(t)\big)\,\big(a(t),
b(t)\big)=\big(d(t)\,a\big(h(t)\big), b\big(h(t)\big)\big)
$$
\begin{Theorem} Denote by $\Rio$ the set of all Riordan arrays. With respect to matrix multiplication, $\Rio$ is a group, where the identity array is $(\augmentation,\augmentation)$, and the multiplicative inverse of any $(\gamma,\alpha)\in\Rio$ is given by $(\gamma,\alpha)^{-1}=(\LL_{\sst\gamma,\alpha},\LL_{\sst\alpha})$.
\end{Theorem}
\begin{proof}
This is a straightforward verification using identity \eqref{eq:Riomul}.
\end{proof}

\subsection{The Sheffer group}

Closely related to Riordan arrays are Sheffer polynomial sequences. Several important
classes of orthogonal polynomial systems: Hermite, Charlier, Laguerre, Meixner of the first and the second kind, are special instances of  \emph{Sheffer sequences}. These polynomial sequences have been deeply studied using finite operator calculus (as shown for instance in Roman's book~\cite{Roman:umbral}).  A treatment of Sheffer sequences under the renewed umbral symbolism can be found in \cite{diNNS:Sheffer,Petrullo:abelidentity,Tay}. We briefly review in this section the connection between Riordan arrays and Sheffer polynomial sequences.
 
A Riordan array $(\gamma,\alpha)$ defines a polynomial sequence
$\entriess_n(x)=\sum_{k=0}^n\,(\gamma,\alpha)_{n,k}\,x^k$. By the fundamental theorem of Riordan arrays, we can write
\begin{equation*}
\entriess_n(x)\simeq\big((\gamma,\alpha)x\big)^n\simeq\big(\gamma+x\punt\beta\punt\der{\alpha}\big)^n \quad,\quad\text{for\ \ all}\quad n\ge 0\,.
\end{equation*}
\noindent In a classical fashion, the sequence $\entriess_n(x)$ satisfies
$1+\sum_{n\geq 1}\entriess_n(x)\frac{t^n}{n!}=f_\gamma(t)\,e^{x\,tf_\alpha(t)}$; namely, it is a Sheffer sequence. We can see from the fundamental theorem that each Riordan array determines a Sheffer sequence and vice versa. In this sense, we say that $\big(\entriess_n(x)\big)$ is the Sheffer sequence of $(\gamma,\alpha)$. Sheffer sequences are endowed with a binary operation called \emph{umbral composition}. More precisely, if $\big(\entriess_n(x)\big)$ and $\big(\entries_n(x)\big)$ are the Sheffer sequences of $(\gamma,\alpha)$ and $(\sigma,\rho)$ respectively, then the umbral composition of $\big(\entriess_n(x)\big)$ and $\big(\entries_n(x)\big)$ is the polynomial sequence $\big([\entriess\,\entries]_n(x)\big)$ defined by
\begin{equation*}
[\entriess\,\entries]_n(x)=\sum_{k=0}^n\,(\gamma,\alpha)_{n,k}\,\entries_k(x)\quad\text{for}\quad n\ge 0\,.
\end{equation*}
\noindent Expanding $\entries_n(x)$, we may easily check that the coefficient of $x^k$ in $[\entriess\,\entries]_n(x)$ is the $(n,k)$-entry of the Riordan array $(\gamma,\alpha)(\sigma,\rho)$. Hence, by \eqref{eq:Riomul}, the polynomial sequence $[\entriess\,\entries]_n(x)$ is the Sheffer sequence of $(\gamma+\sigma\punt\beta\punt\der{\alpha}\,,\,\alpha+\rho\punt\beta\punt\der{\alpha})$. Also, it readily follows that the Sheffer sequence $x^n$ of $(\augmentation,\augmentation)$ is the identity element for the umbral composition and that $\big(\entriess_n(x)\big)$ is the inverse of $\big(\entries_n(x)\big)$ if and only if $(\sigma,\rho)=(\LL_{\sst\gamma,\alpha},\LL_{\sst\alpha})$. Thus, the set $\Sheff$ of all  Sheffer sequences endowed with the umbral composition is a group, and moreover, it follows from the previous discussion that $\Sheff$ and $\Rio$ are isomorphic.

\subsection{Riordan subgroups}
\label{sse:riordansubgroups}

\begin{table}[ht]
\centering
\begin{tabular}{lcclclcl}
               \toprule
                \textbf{Subgroup} & $(\gamma,\alpha)$ & & $(\gamma,\alpha)_{n,k}$ & & $(\gamma,\alpha)^\inv$ & & \hspace{2cm}Product\\
               \midrule
               Appell & $(\gamma,\augmentation)$ & & $\Binom{n}{k}\gamma^{n-k}$ & & $(-1\punt\gamma,\augmentation)$ & & $(\gamma,\augmentation)(\sigma,\augmentation)=(\gamma+\sigma,\augmentation)$ \\[1em]
               Associated & $(\augmentation,\alpha)$ & & $\Binom{n}{k}\big(k\punt\alpha\big)^{n-k}$  & &  $(\augmentation,\LL_{\sst\alpha})$ & & $(\augmentation,\alpha)(\augmentation,\rho)=(\augmentation,\alpha+\rho\punt\bell\punt\der{\alpha})$ \\[1em]
               Bell & $(\alpha,\alpha)$ & &  $\Binom{n}{k}\big((k+1)\punt\alpha\big)^{n-k}$ & & $(\LL_{\sst\alpha},\LL_{\sst\alpha})$ & &  $(\alpha,\alpha)(\sigma,\sigma)=(\alpha+\sigma\punt\bell\punt\der{\alpha},\alpha+\sigma\punt\bell\punt\der{\alpha})$ \\
               \bottomrule
\end{tabular}
\caption{Some Riordan subgroups}\label{tb:subgroups}
\vspace{-0.8cm}
\end{table}

Riordan subgroups can be neatly described in umbral terms. A sample of them is shown in Table \ref{tb:subgroups}. The umbral notation allows for easy verifications. For example, a straightforward use of \eqref{eq:Riomul} yields $(\gamma,\alpha)=(\gamma,\augmentation)(\augmentation,\alpha)$, showing that any Riordan array is the product of an Appell array with an Associated array. Similarly, one can write $(\gamma,\alpha)=(\gamma+(-1)\punt\alpha,\augmentation)(\alpha,\alpha)$,  which shows that any Riordan array is also the product of an Appell array with a Bell array. Likewise, the fundamental theorem of Riordan arrays applied to each of the subgroups shown in Table \ref{tb:subgroups} yields: for the Appell subgroup, 
\[(\gamma,\varepsilon)\eta\equal \gamma+\eta \punt\bell\punt\der{\augmentation}\equal \gamma+\eta \punt\bell\punt\singleton\equal \gamma+\eta\,.\]

\noindent For the Associated subgroup, 
\[(\augmentation,\alpha)\eta\equal\augmentation+\eta\punt\beta\punt\der{\alpha}\equal\eta\punt\beta\punt\der{\alpha}\,.\]

\noindent Finally, for the Bell subgroup,  
\[(\alpha,\alpha)\eta\equal\alpha+\eta\punt\beta\punt\der{\alpha}\,.\]

\subsection{Recursive properties}\label{se:recursions}
Riordan arrays are also characterized by means of recurrence relations. The recurrence relations   stated next (Theorem~\ref{th:recursions1}) are direct consequences of  Definition~\ref{de:riordanarray} and the umbral Abel identity \eqref{eq:abel-identity}. Since these formulas  are particular cases of recurrence relations satisfied by Riordan arrays in a more general context, we postpone their proofs to Section \ref{se:genrio}.
\begin{Theorem}\label{th:recursions1}
A Riordan array $(\gamma,\alpha)$ satisfies the following recursive properties:
\begin{equation}\label{eq:colrec}(\gamma,\alpha)_{n,k}\simeq \frac{n}{k}\,\sum_{i=0}^{n-k}\binom{n-1}{i}    \,\alpha^i\,(\gamma,\alpha)_{n-1-i,k-1}\quad,\quad n\geq k\ge 1\,.\end{equation}

\begin{equation}\label{eq:rowrec}(\gamma,\alpha)_{n,k}\simeq \frac{n}{k}\,\sum_{i=0}^{n-k}\binom{k-1+i}{i}\,\K_{\sst\alpha}^i\,(\gamma,\alpha)_{n-1,k-1+i}\quad,\quad n\geq k\ge 1\,.\end{equation}

\begin{equation}\label{eq:rowrec2}(\gamma,\alpha)_{n,k}\simeq \binom{n}{k}\,\sum_{i=0}^{n-k}\,(k\punt\K_{\sst\alpha})^i\,(\gamma,\alpha)_{n-k,i}\quad,\quad n\ge k\ge 1\,.\end{equation}
\end{Theorem}
\begin{Remark} Traditional formulas are obtained by applying the evaluation map $E$ on the identities \eqref{eq:colrec}, \eqref{eq:rowrec} and \eqref{eq:rowrec2}. Thus, the umbral equivalence $\simeq$ is substituted by the ordinary equality $=$ and $\alpha^i$, $\K_{\sst\alpha}^i$ and $(k\punt\K_{\sst\alpha})^i$ are replaced with their corresponding moments, say $h_i\simeq\alpha^i$, $a_i\simeq\K_{\sst\alpha}^i$ and $a_i^{\sst(k)}\simeq(k\punt\K_{\sst\alpha})^i$, where $a_i^{\sst(k)}\simeq i![t^i](f_{\K_{\sst\alpha}}(t))^k=i![t^i]\big(f_\alpha\big((tf_\alpha(t))^\invumb\big)\big)^k$. The sequence $(a_i)=(a_i^{\sst(1)})$ is nothing but the analogue of the $A$-sequence of Rogers \cite{Rog} for exponential Riordan arrays.
\end{Remark} 
\section{Examples}
\label{se:examples}
The purpose of this section is to analyze concrete examples  of classical Riordan arrays and their corresponding Sheffer sequences by using the umbral syntax.
\subsection*{Pascal array} The Pascal array $\Pascal=(\unity,\augmentation)$ (written as $(e^t,t)$ in classical terms) satisfies
\[\Pascal_{n,k}\simeq\binom{n}{k}(\unity+k\punt \augmentation)^{n-k}\simeq\binom{n}{k}(\unity)^{n-k}\simeq\binom{n}{k}.\]
\noindent Since $\Pascal$ is an Appell array, we have
$\Pascal\unity\equal \unity+\unity\equiv 2$, so that, after evaluation, we recover the well-known identity $2^n=\sum_{k=0}^n\binom{n}{k}$. Note that the identity $\Pascal x\equal 1+x$ encodes the formula $(x+1)^n=\sum_{k=0}^n \binom{n}{k}x^k$.
\subsection*{Stirling arrays} The Stirling array of the second kind, traditionally written as $(1,e^t-1)$, is encoded by $\StirlingII=(\augmentation,-1\punt\bernoulli)$, where $\bernoulli$ is the Bernoulli umbra (see Table \ref{tb:keyumbrae}). It is an Associated array. Its entries $\StirlingII_{n,k}$ are the classical Stirling numbers of the second kind, usually denoted by $S(n,k)$. Some known properties of this array can be derived from the umbral setting as follows.
\begin{enumerate}[itemsep=1ex,leftmargin=0.8cm]
\item Via generating functions, it is easy to see that $\der{(-1\punt\bernoulli)}\equiv\unity$. It follows from Theorem \ref{th:FTRA} that $\StirlingII\unity\equal\augmentation+\unity\punt\beta\punt\der{(-1\punt\bernoulli)}\equiv \beta$. After evaluation, the sequence of row sums of this array gives the sequence of Bell numbers; that is, $\sum_{k=0}^nS(n,k)=B_n$, as expected.

\item Let $\big(\phi_n(x)\big)$ be the sequence of polynomials in $x$ represented by $x\punt\bell$. Since $\StirlingII x\equal x\punt\beta$, the Sheffer sequence of $\StirlingII$ is precisely $\big(\phi_n(x)\big)$. After evaluation, we obtain $\phi_n(x)=\sum_{k=0}^n S(n,k)\,x^k$. The $\phi_n(x)$ are called exponential polynomials~\cite{Roman:umbral}.

\item Theorem \ref{th:recursions1} allows us to obtain recurrence formulas for the Stirling numbers of second kind. Since $\der{(-1\punt\bernoulli)}\equiv\unity$, we have $\der{(-1\punt\bernoulli)}^i\simeq i(-1\punt\bernoulli)^{i-1}\simeq 1$ for all $i\ge 1$. Hence, $(-1\punt \bernoulli)^i\simeq\frac{1}{i+1}$. From identity \eqref{eq:colrec} and after evaluation, we obtain (compared with \cite{Comtet})
\[S(n,k)=\frac{n}{k}\sum_{i=0}^{n-k}\binom{n-1}{i}\,\frac{1}{(i+1)}\,S(n-1-i,k-1)\,.\]

\noindent Also, since $\K_{-1\punt\bernoulli}\equiv-1\punt\bernoulli\punt\beta\punt\der{(-1\punt\bernoulli)}^\invumb$, and $\der{(-1\punt\bernoulli)}^\invumb\equiv\unity^\invumb\equiv\chi\punt\chi$, we have  $\K_{-1\punt\bernoulli}\equiv-1\punt\bernoulli\punt\chi$. It is easy to check that the g.f. of $-1\punt\bernoulli\punt\chi$ is $f_{-1\punt\bernoulli\punt\chi}(t)=\frac{t}{\log(1+t)}$. The rational numbers  represented by the moments of $-1\punt\bernoulli\punt\chi$ are known as Cauchy numbers of the first type \cite{Comtet}. Following \cite{MSV:Cauchy}, we denote these numbers by $\mathscr{C}_i$. From \eqref{eq:rowrec}, we get the following recurrence identity that relates Stirling numbers of the second kind and Cauchy numbers of the first kind (compare with \cite[Theorem 3.1]{MSV:Cauchy}).
\[S(n,k)=\frac{n}{k}\sum_{i=0}^{n-k}\binom{k-1+i}{i}\mathscr{C}_i\,S(n-1-i,k-1)\,.\]
\noindent The moments $(k\punt\K_{-1\punt\bernoulli})^i\simeq\mathscr{C}_i^{\sst(k)}=i![t^i]\big(\frac{t}{\log(1+t)}\big)^k$ are known as generalized Cauchy numbers of the first type. They show up in the following recurrence relation obtained from \eqref{eq:rowrec2},
\[S(n,k)=\binom{n}{k}\sum_{i=0}^{n-k}\mathscr{C}^{\sst (k)}_i\,S(n-k,i)\,.\]
\end{enumerate}

\noindent Let us now denote by $\StirlingI$ the array whose entries $\StirlingI_{n,k}$ are the Stirling numbers of the first kind $s(n,k)$. In traditional terms, we have $\StirlingI=(1,\log(1+t))$. It is known that this array is the inverse of $\StirlingII$. Since $\StirlingII$ is an Associated array and $\LL_{\sst-1\punt\bernoulli}\equiv-1\punt\K_{\sst -1\punt\bernoulli}\equiv\bernoulli\punt\chi$, it follows that $\StirlingI=(\augmentation, \LL_{\sst-1\punt\bernoulli})=(\augmentation,\bernoulli\punt\chi)$. As shown above, known properties involving the numbers $s(n,k)$ can be obtained through the umbral setting as follows.
\begin{enumerate}[itemsep=1ex,leftmargin=0.8cm]
\item Via g.f.'s we have $\der{(\bernoulli\punt\chi)}\equiv\chi\punt\chi$. By Theorem \ref{th:FTRA}, we obtain $\StirlingI\unity\equal\unity\punt\bell\punt\der{(\bernoulli\punt\chi)}\equiv\unity\punt\bell\punt\singleton\punt\singleton\equal\chi$. As $\StirlingI_{0,0}\simeq 1$, $\StirlingI_{1,1}\simeq 1$ and $\StirlingI_{n,0}\simeq 0$ for $n\geq 1$, it then follows after evaluation that
\[0=\sum_{k=0}^n s(n,k)\quad\text{for}\quad n\ge 2\,.\]
\item  Note that for any $x\in R\cup A$, we have $f_{x\punt\singleton}(t)\simeq(1+t)^x=1+\sum_{n\ge 1}(x)_n\frac{t^n}{n!}$, where $(x)_n$ is the lower factorial of $x$. It follows that $(x\punt\singleton)^n\simeq(x)_n$ for all $n\ge 0$. Since $\StirlingI x\equiv x\punt\bell\punt \der{(-1\punt\bernoulli)}\equal x \punt\bell\punt\singleton\punt\singleton\equal x\punt\singleton$, the Sheffer sequence of $\StirlingI$ is the sequence of falling factorials $\big((x)_n\big)$, that is
\[(x)_n=\sum_{k=0}^n s(n,k)\,x^k\,.\]
\item Since $\der{(\bernoulli\punt\chi)}^n\simeq(\chi\punt\chi)^n\simeq(\singleton)_n\simeq (-1)^{n-1}(n-1)!$, identity \eqref{eq:colrec} specializes into
\[s(n,k)=\frac{n}{k}\sum_{i=0}^{n-k}\binom{n-1}{i}\frac{(-1)^i}{(i+1)}\,i!\, s(n-1-i,k-1)\,.\]
Also, since $\K_{\bernoulli\punt\singleton}\equiv\bernoulli$ and $\bernoulli^i\simeq b_i$ (the $i^{\rm{th}}$ Bernoulli number), identity \eqref{eq:rowrec} gives
\[s(n,k)=\frac{n}{k}\sum_{i=0}^{n-k}\binom{k-1+i}{i}\,b_i\,s(n-1-i,k-1)\,.\]
\noindent Finally, for every $k\ge 0$, the umbra $k\punt\bernoulli$ represents the sequence of generalized Bernoulli numbers $b_i^{\sst(k)}$; that is, $(k\punt\bernoulli)^i\simeq b_i^{\sst(k)}=i! [t^i] \big(\frac{t}{e^t-1}\big)^k$. It then follows from \eqref{eq:rowrec2} that
\[s(n,k)=\binom{n}{k}\sum_{i=0}^{n-k}\,b_i^{\sst(k)}\,s(n-k,i)\,.\]
\end{enumerate} 
\section{Generalized Riordan arrays}\label{se:genrio}

As seen in the Introduction, Riordan arrays can be described in several ways according to the type of their generating functions and the manner in which coefficients are extracted from them. Thus, the notion of a generalized Riordan array with respect to a given sequence of nonzero numbers arises \cite{WW}. By means of the umbral syntax, generalized Riordan arrays are handled in a neat and unified way, extending the discussion on exponential Riordan arrays given in Section \ref{se:riordan}. In this regard, as a direct consequence of the umbral Abel identity \eqref{eq:abel-identity}, we obtain an important umbral polynomial equivalence (Theorem \ref{th:mother-rec}) from which a novel non-recursive formula for Riordan arrays is derived (Theorem \ref{th:nonrec}). Specializations of Theorem \ref{th:nonrec} allows us to recover some known recurrence relations of Riordan arrays (Theorems \ref{th:rechor} and \ref{th:recver}) and unveil a new recurrence relation (Theorem \ref{th:recdiff}).

\begin{Definition}\label{de:genriordanarray}
Let $(c_n)$ be a sequence of nonzero numbers such that $c_0=1$. Let $\omega$ be an umbra whose moments are $\omega^n\simeq\frac{n!}{c_n}$, for all $\ n\ge 0$. Given two umbrae $\alpha$ and $\gamma$, we denote by $\prescript{}{\omega}{(\gamma,\alpha)}$ the infinite lower triangular matrix whose entries $\prescript{}{\omega}{(\gamma,\alpha)}_{n,k}$ satisfy
\begin{equation}\label{def:gen_rio}
\prescript{}{\omega}{(\gamma,\alpha)}_{n,k}\simeq \frac{c_n}{c_k}\,\frac{(\gamma+k\punt\alpha)^{n-k}}{(n-k)!}\,,\quad\text{for}\quad n\geq k\geq 0.
\end{equation}
The matrix $\prescript{}{\omega}{(\gamma,\alpha)}$ is called $\omega$-Riordan array and it is said to be generated by $\alpha$ and $\gamma$. In general, for $p,q\in R[A]$, we define  $\omega$-Riordan arrays $\prescript{}{\omega}{(p,q)}$ in an analogous way. Observe that $\prescript{}{\omega}{(\gamma,\alpha)}_{n,n}=1$. This means that we are dealing with normalized arrays.
\end{Definition}

\noindent Letting $e^{\gamma t}\simeq d(t)=\sum_{n\ge 0} d_n \frac{t^n}{c_n}$ and $te^{\alpha t}\simeq h(t)=\sum_{n\ge 1} h_n \frac{t^n}{c_n}$, with $d(0)=1$, $h(0)=0$, and $h'(0)=1$, we get the traditional presentation
\begin{equation*}
r_{n,k} = \prescript{}{\omega}{(\gamma,\alpha)}_{n,k} = E\left[\frac{c_n}{c_k}\,\frac{(\gamma+k\punt\alpha)^{n-k}}{(n-k)!}\right] = c_n\left[t^n\right] d(t)\,\frac{h(t)^k}{c_k}\,.
\end{equation*}

\noindent Exponential Riordan arrays are obtained by setting  $\omega\equiv \unity$ (that is, $c_n=n!$), in which case we simply write $\prescript{}{\unity}{(\gamma,\alpha)}$ as $(\gamma,\alpha)$. Ordinary Riordan arrays are obtained by setting  $\omega\equiv\bar{\unity}$ (that is, $c_n=1$), where $\bar{\unity}$ is the boolean unity umbra whose moments are $\bar{\unity}^n\simeq n!$ (see Table~\ref{tb:keyumbrae}).  We denote by $\prescript{}{\omega}{\Rio}$  the set of all $\omega$-Riordan arrays. All known definitions and properties regarding (exponential) Riordan arrays readily extend to $\omega$-Riordan arrays. Thus, for instance, it follows that $\prescript{}{\omega}{\Rio}$ is a group and the natural $\omega$-versions of the subgroups of $\Rio$ are in turn subgroups of $\prescript{}{\omega}{\Rio}$. As a sample, we give next a statement whose straightforward umbral proof is left to the reader.
\begin{Theorem}\label{thm:omega_arrays}
Let $\omega$ be an umbra with nonzero moments and let $(\gamma,\alpha)\in \Rio$. Then
\[\prescript{}{\omega}{(\gamma,\alpha)}\,\prescript{}{\omega}{(\sigma,\rho)}=
\prescript{}{\omega}{[(\gamma,\alpha)(\sigma,\rho)]}\ \ {\text and}\ \ [\prescript{}{\omega}{(\gamma,\alpha)}]^{\inv}
=\prescript{}{\omega}{(\gamma,\alpha)}^{\inv}=\prescript{}{\omega}{(\LL_{\sst\gamma,\alpha},\LL_{\sst\alpha})}.\]
\end{Theorem}
\noindent Furthermore, the fundamental theorem in the context of $\omega$-Riordan arrays is stated as follows. 
\begin{Theorem}[Fundamental Theorem of $\omega$-Riordan arrays]\label{th:omegaFTRA}
For all umbrae $\alpha, \gamma$ and $\eta$, we have
\begin{equation}\label{def:omega_action}
\omega\big[\prescript{}{\omega}{(\alpha,\gamma)}\big]\eta\equiv(\gamma,\alpha)\omega\eta
\equiv \gamma+\omega\eta\punt\bell\punt\der{\alpha}\,.
\end{equation}
\end{Theorem}

\noindent If we choose $\eta\equiv\unity$ in identity \eqref{def:omega_action}, we can see that the row sums of $\prescript{}{\omega}{(\gamma,\alpha)}$ are given by
\begin{equation}\label{eq:gen_rowsum}
\sum_{k=0}^n\,\prescript{}{\omega}{(\gamma,\alpha)}_{n,k}\simeq
\frac{c_n\,(\gamma+\omega\punt\beta\punt\der{\alpha})^n}{n!}\,.
\end{equation}
\noindent In the same manner as Sheffer sequences are defined from exponential Riordan arrays, we define $\omega$-Sheffer sequences out of $\omega$-Riordan arrays by
$$\entriess^\omega_n(x)=\sum_{k=0}^n\,\prescript{}{\omega}{(\gamma,\alpha)}_{n,k}\,x^k. $$
\noindent It immediately follows that for every Sheffer sequence $(\entriess_n(x))$, there exists a unique polynomial sequence $(\entriess^\omega_n(x))$ such that
\begin{equation}\label{def:gen_sheff}\omega^n\, \entriess^\omega_n(x)\simeq \entriess_n(\omega x),\quad n\geq 0.\end{equation}
\noindent Since $\entriess_n(\omega x)\simeq \big(\gamma+\omega x \punt\bell\punt\der{\alpha}\big)^n$, it follows from \eqref{def:gen_sheff} that
$\displaystyle \entriess^\omega_n(x)\simeq \frac{c_n\,(\gamma+\omega x\punt\beta\punt\der{\alpha})^n}{n!}$.
\noindent We denote by $\prescript{}{\omega}{\Sheff}$ the set of all polynomial sequences $(\entriess^\omega_n(x))$ and refer to it as the $\omega$-Sheffer set.

\medskip

In order to state recurrence relations for $\omega$-Riordan arrays, consider the following simple umbral trick: $\gamma+k\punt\alpha\equal\gamma+(k-m)\punt\alpha+m\punt\alpha$, where $k$ and $m$ are any integers. The following main result is a straightforward consequence of the umbral Abel identity \eqref{eq:abel-identity}.

\begin{Theorem}\label{th:mother-rec}
For all $n,k\in\Z$ such that $n-k\ge 0$, we have
\begin{equation*}\label{eq:keyforrec}
\begin{array}{lcl}
(\gamma+k\punt\alpha)^{n-k} &\simeq& \dsum_{i=0}^{n-k}\binom{n-k}{i}
\big(\gamma+(k-m)\punt\alpha+i\punt\lambda\big)^{n-k-i}\,\,
m\punt\alpha\big(m\punt\alpha+(-i)\punt\lambda\big)^{i-1}\\[2em]
&\simeq&(n-k)! \, \dsum_{i=0}^{n-k} \frac{\big(\gamma+(k-m)\punt\alpha+i\punt\lambda\big)^{n-k-i}}{(n-k-i)!}\,\,
\frac{m\punt\alpha\big(m\punt\alpha+(-i)\punt\lambda\big)^{i-1}}{i!}\,\,.
\end{array}
\end{equation*}
\end{Theorem}
\noindent Recall that, on the one hand, we have $\K_{\sst m\punt\alpha,\lambda}^i\simeq m\punt\alpha\big(m\punt\alpha+(-i)\punt\lambda\big)^{i-1}$, for all $i\ge 1$. On the other hand, we also have $\K_{\sst{m\punt\alpha,\lambda}}\equal m\punt\alpha\punt\bell\punt\der{\lambda}^\invumb$. We then conclude that  $\K_{\sst{m\punt\alpha,\lambda}}\equal m\punt\K_{\sst\alpha,\lambda}$. It readily follows from Definition \ref{de:genriordanarray} and Theorem \ref{th:mother-rec}, the next key result.
\begin{Theorem}\label{th:nonrec}
Given any umbra $\lambda$ and any integers $m,n,k$, with $n\ge k$, it holds
\begin{equation}\label{eq:nonrec}
\prescript{}{\omega}{(\gamma,\alpha)}_{n,k}\simeq
\frac{c_n}{c_k}\dsum_{i=0}^{n-k} \frac{
\big(m\punt\K_{\sst\alpha,\lambda}\big)^{i}}{i!}\,\, \frac{
\big(\gamma+(k-m)\punt\alpha+i\punt\lambda\big)^{n-k-i}}{(n-k-i)!}\,.
\end{equation}
\end{Theorem}

\noindent Identity \eqref{eq:nonrec} is equivalent to 

\begin{equation}\label{eq:nonrec-lambda}
\prescript{}{\omega}{(\gamma,\alpha)}_{n,k}\simeq \frac{c_n}{c_k\,c_{n-k}}
\dsum_{i=0}^{n-k} c_i \,\frac{
\big(m\punt\K_{\sst\alpha,\lambda}\big)^{i}}{i!}
\,\, \prescript{}{\omega}{(\gamma+(k-m)\punt\alpha,\lambda)}_{n-k,i}\,.
\end{equation}

\noindent In particular, the $(n,k)$-entry of any exponential Riordan array is given by
\begin{equation}\label{eq:nonrec-exp}
(\gamma,\alpha)_{n,k}\simeq\binom{n}{k}\dsum_{i=0}^{n-k} \binom{n-k}{i}(m\punt\K_{\sst\alpha,\lambda})^i \, \big(\gamma+(k-m)\punt\alpha+i\punt\lambda\big)^{n-k-i}\,.
\end{equation}

\begin{Remark} Theorem \ref{th:nonrec} is important not only for the recurrence relations (Theorems \ref{th:rechor}, \ref{th:recver} and \ref{th:recdiff}) that we shall see are derived from it, but also because it provides a way of expressing  the $(n,k)$-entry of any $\omega$-Riordan array as a special weighted sum of entries of a related $\omega$-Riordan array parametrized by $\lambda\in A$. Formula~\eqref{eq:nonrec-lambda} is, to the best of our knowledge, new. Remarkable combinatorial identities as the ones shown in Examples \ref{ex:stirling-bernoulli-cauchy} and \ref{ex:stirling-bernoulli} are direct  consequences of this formula. 
\end{Remark}

\begin{Example}\label{ex:stirling-bernoulli-cauchy}
Consider the Stirling array of the second kind $\StirlingII=(\augmentation,-1\punt\bernoulli)$ and set $\lambda\equal-1\punt\bernoulli$ in \eqref{eq:nonrec-exp}. Recall that $\K_{\sst-1\punt\bernoulli,-1\punt\bernoulli}\equal-1\punt\bernoulli\punt\chi$.  Hence, for any $n\ge k\ge 0$ and $m\in\Z$, we have
\begin{equation*}
S(n,k)\simeq\binom{n}{k}\dsum_{i=0}^{n-k} \binom{n-k}{i} \big((-m)\punt\bernoulli\punt\chi\big)^i \, \big((m-k-i)\punt\bernoulli\big)^{n-k-i}\, .
\end{equation*}
Since $\big((-m)\punt\bernoulli\punt\chi\big)^i \simeq \mathscr{C}_i^{\sst(m)}$  and $\big((m-k-i)\punt\bernoulli\big)^{n-k-i}\simeq b_{n-k-i}^{\sst(m-k-i)}$ (these are respectively the generalized Cauchy and Bernoulli numbers that were used in Section \ref{se:examples}), we get the following combinatorial formula
\begin{equation*}
S(n,k) = \binom{n}{k}\dsum_{i=0}^{n-k} \binom{n-k}{i} \mathscr{C}_i^{\sst(m)} \, b_{n-k-i}^{\sst(m-k-i)} \,.
\end{equation*}
\end{Example}

\noindent Now, set $\lambda\equal\augmentation$ in \eqref{eq:nonrec-lambda}. We have $\der{\augmentation}\equal\singleton$ and  $\singleton^\invumb\equal\singleton$; hence $\K_{\sst\alpha,\augmentation}\equal\alpha$. The next identity follows.

\begin{Corollary}\label{co:nonrec}
The entries of any $\omega$-Riordan array satisfy
\begin{equation}\label{eq:nonrec-aug}
\prescript{}{\omega}{(\gamma,\alpha)}_{n,k}\simeq \frac{c_n}{c_k\,c_{n-k}}
\dsum_{i=0}^{n-k} c_i \,\frac{
\big(m\punt\alpha\big)^{i}}{i!}\,\, \prescript{}{\omega}{(\gamma+(k-m)\punt\alpha,\augmentation)}_{n-k,i}\,.
\end{equation}

\noindent In traditional terms, formula \eqref{eq:nonrec-aug} is equivalent to
\begin{equation*}
r_{n,k}=\frac{c_n}{c_k\,c_{n-k}}\dsum_{i=0}^{n-k} c_i\frac{a_i^{\sst(k)}}{i!}\,u_{n-k,i}\,\,\,,
\end{equation*}
where $(m\punt\alpha)^i\simeq a_i^{\sst(m)}=i![t^i]\big(f_\alpha(t)\big)^m$ and
\begin{equation*}
\begin{array}{lclcl}
\prescript{}{\omega}{(\gamma,\alpha)}_{n,k} &=& r_{n,k} &=& c_n[t^n]f_\gamma(t)\frac{(tf_\alpha(t))^k}{c_k}\quad\text{and}\\[1em]
\prescript{}{\omega}{(\gamma+(k-m)\punt\alpha,\augmentation)}_{n-k,i} &=& u_{n-k,i} &=& c_{n-k}[t^{n-k}]f_\gamma(t)f_\alpha(t)^{k-m}\frac{t^i}{c_i}\,.
\end{array}
\end{equation*}

\noindent Moreover, setting $m=k\ge 0$ in \eqref{eq:nonrec-aug}, we obtain
\begin{equation*}
\prescript{}{\omega}{(\gamma,\alpha)}_{n,k}\simeq \frac{c_n}{c_k\,c_{n-k}}
\dsum_{i=0}^{n-k} c_i \,\frac{
\big(k\punt\alpha\big)^{i}}{i!}\,\, \prescript{}{\omega}{(\gamma,\augmentation)}_{n-k,i}\,.
\end{equation*}
\end{Corollary}

\begin{Example}\label{ex:stirling-bernoulli} Consider again $\StirlingII=(\augmentation,-1\punt\bernoulli)$ and formula \eqref{eq:nonrec-aug}. For any $n\ge k\ge 0$ and $m\in\Z$, we have 
\begin{equation*}
S(n,k)\simeq\binom{n}{k}\dsum_{i=0}^{n-k} \binom{n-k}{i} \big((-m)\punt\bernoulli\big)^i \, \big((m-k)\punt\bernoulli\big)^{n-k-i}\simeq \binom{n}{k}\big((-k)\punt\bernoulli\big)^{n-k}\simeq\binom{n}{k} b_{n-k}^{\sst(-k)}\, .
\end{equation*}
\noindent In traditional terms, the identity above is written as (see \cite[Proposition 9.1]{RT:classicalumbral})
\begin{equation*}
S(n,k)=\binom{n}{k}  b_{n-k}^{\sst(-k)}\,.
\end{equation*}
\end{Example}

By specializing $\lambda\equal\alpha$ in identity~\eqref{eq:nonrec-lambda} and setting $m=k\ge 0$, the following linear \emph{horizontal} recurrence relation for $\omega$-Riordan arrays is obtained.

\begin{Corollary}\label{co:rec}
Any $\omega$-Riordan array satisfy the following recurrence relation,
\begin{equation*}
\prescript{}{\omega}{(\gamma,\alpha)}_{n,k} \simeq
\frac{c_n}{c_k\,c_{n-k}}\dsum_{i=0}^{n-k} c_i \,\frac{
\big(k\punt\K_{\sst\alpha}\big)^{i}}{i!}
\,\, \prescript{}{\omega}{(\gamma,\alpha)}_{n-k,i}\quad,\quad n\ge k\ge 0.
\end{equation*}
\end{Corollary}

\noindent Furthermore, by setting $\lambda\equal\alpha$ in identity~\eqref{eq:nonrec}, Corollary \ref{co:rec} is obtained as a particular case of a more general horizontal recurrence relation for $\omega$-Riordan arrays.

\begin{Theorem}\label{th:rechor}
For any integer $m$ such that $n\ge m$, it holds
\begin{equation*}
\prescript{}{\omega}{(\gamma,\alpha)}_{n,k}\simeq \frac{c_n}{c_k\,c_{n-m}}
\dsum_{i=0}^{n-k} c_{k-m+i} \,\frac{
\big(m\punt\K_{\sst\alpha}\big)^{i}}{i!}\,\, \prescript{}{\omega}{(\gamma,\alpha)}_{n-m,k-m+i}\,.
\end{equation*}
\end{Theorem}

\begin{proof}
We can write $n-k-i=(n-m)-(k-m+i)$ so that
\begin{equation*}
\begin{array}{lcl}
\prescript{}{\omega}{(\gamma,\alpha)}_{n,k} &\simeq&
\dfrac{c_n}{c_k}\dsum_{i=0}^{n-k} \frac{
\big(m\punt\K_{\sst\alpha}\big)^{i}}{i!}\,\, \frac{
\big(\gamma+(k-m+i)\punt\alpha\big)^{(n-m)-(k-m+i)}}{\big((n-m)-(k-m+i)\big)!}\\[2em]
&\simeq& \dfrac{c_n}{c_k}\dsum_{i=0}^{n-k} \frac{
\big(m\punt\K_{\sst\alpha}\big)^{i}}{i!}\,\,\frac{c_{k-m+i}}{c_{n-m}}
\,\,\prescript{}{\omega}{(\gamma,\alpha)}_{n-m,k-m+i}\,\,.
\end{array}
\end{equation*}
\end{proof}

\begin{Remark} Theorem \ref{th:rechor} extends identities \eqref{eq:rowrec} and \eqref{eq:rowrec2} of Theorem \ref{th:recursions1} to generalized Riordan arrays. Setting $m=1$ and applying the the linear functional $E$ in Theorem \ref{th:rechor}, we recover a shifted version of the recurrence relation stated in \cite[Theorem 5.2]{WW}. Moreover, by setting $\omega\equal\bar{\unity}$, we also recover the shifted horizontal recurrence relation of \cite[see Theorem 3.1 and Formula (3.3) therein]{LMMS:identities}. More explicitly, we have $\frac{(m\punt\K_{\sst\alpha})^i}{i!}\simeq\frac{a_i^{\sst(m)}}{i!}$, where $a_i^{\sst(m)}=i![t^i]\big(f_\alpha\big((tf_\alpha(t))^\invumb\big)\big)^m$. The sequence $\nicefrac{a_i^{\sst(m)}}{i!}$ is a generalization of the $A$-sequence of (ordinary) Riordan arrays.
\end{Remark}

Now, setting $\lambda\equal\augmentation$ in identity~\eqref{eq:nonrec}, we obtain the following \emph{vertical} recurrence relation for $\omega$-Riordan arrays.

\begin{Theorem}\label{th:recver}
For any integer $m$ such that $k\ge m$, it holds
\begin{equation*}
\prescript{}{\omega}{(\gamma,\alpha)}_{n,k}\simeq \frac{c_n\,c_{k-m}}{c_k}
\,\dsum_{i=0}^{n-k} \frac{1}{c_{n-m-i}}\,
\frac{\big(m\punt\alpha\big)^{i}}{i!}\,\, \prescript{}{\omega}{(\gamma,\alpha)}_{n-m-i,k-m}\,.
\end{equation*}
\end{Theorem}

\begin{proof}
Simply write $n-k-i=(n-m-i)-(k-m)$ so that
\begin{equation*}
\begin{array}{lcl}
\prescript{}{\omega}{(\gamma,\alpha)}_{n,k} &\simeq&
\dfrac{c_n}{c_k}\dsum_{i=0}^{n-k} \frac{
\big(m\punt\alpha\big)^{i}}{i!}\,\, \frac{
\big(\gamma+(k-m)\punt\alpha\big)^{(n-m-i)-(k-m)}}{\big((n-m-i)-(k-m)\big)!}\\[2em]
&\simeq& \dfrac{c_n\,c_{k-m}}{c_k}\dsum_{i=0}^{n-k} \frac{
\big(m\punt\alpha\big)^{i}}{i!}\,\,\frac{1}{c_{n-m-i}}
\,\,\prescript{}{\omega}{(\gamma,\alpha)}_{n-m-i,k-m}\,\,.
\end{array}
\end{equation*}
\end{proof}

\begin{Remark} Theorem \ref{th:recver} extends identity \eqref{eq:colrec} of Theorem \ref{th:recursions1} to generalized Riordan arrays. It also extends the vertical recurrence relation of ordinary Riordan arrays stated in \cite[see Theorem 3.2 and the shifted version (3.3) therein]{LMMS:identities}. Note also that, by setting $m=1$ and applying the evaluation map $E$ in Theorem \ref{th:recver}, we obtain a different identity compared with the recurrence relation for generalized Riordan arrays stated in \cite[Theorem 5.5]{WW}.
\end{Remark}

As a further example of the power of the umbral symbolic approach, consider  $\lambda\equal-1\punt\alpha$ in identity \eqref{eq:nonrec}. We then obtain the following novel recurrence relation for $\omega$-Riordan arrays.

\begin{Theorem}\label{th:recdiff}
For any integer $m$ such that $2k-n\ge m$, it holds
\begin{equation*}
\prescript{}{\omega}{(\gamma,\alpha)}_{n,k}\simeq \frac{c_n}{c_k}
\,\dsum_{i=0}^{n-k} \frac{c_{k-m-i}}{c_{n-m-2i}}\,
\frac{\big(-m\punt\K_{\sst-1\punt\alpha}\big)^{i}}{i!}\,\, \prescript{}{\omega}{(\gamma,\alpha)}_{n-m-2i,k-m-i}\,.
\end{equation*}
\end{Theorem}

\begin{proof}
The equivalence follows by similar considerations as those used in the proofs of Theorem \ref{th:rechor} and Theorem \ref{th:recver}, noting that we can write $n-k-i=(n-m-2i)-(k-m-i)$ and $-m\punt\K_{\sst-1\punt\alpha}\equal m\punt\K_{\sst\alpha,-1\punt\alpha}$. The condition $2k-n\ge m$ is required to guarantee that $k-m-i\ge 0$ and $n-m-2i\ge 0$, in order to consider the corresponding lower triangular entries $\prescript{}{\omega}{(\gamma,\alpha)}_{n-m-2i,k-m-i}$, for $0\le i \le n-k$.
\end{proof}

In classical terms, the umbral identity given in Theorem \ref{th:recdiff} is written as
\begin{equation}\label{eq:recdiff-classical}
r_{n,k}=\frac{c_n}{c_k}\dsum_{i=0}^{n-k}\frac{c_{k-m-i}}{c_{n-m-2i}}\frac{a_i^{\sst(m)}}{i!} \, 
r_{n-m-2i,k-m-i}\,\,,
\end{equation}
where $\big(-m\punt\K_{\sst-1\punt\alpha}\big)^i\simeq a_i^{\sst(m)}=i! [t^i] \big(f_\alpha\big(\,\,(\nicefrac{t}{f_\alpha(t)})^\invumb\big)\big)^m$.

\begin{Example}
As an ordinary Riordan array, the Pascal array $\Pascal$ is written as $\prescript{}{\bar{\unity}}{(\bar{\unity},\bar{\unity})}$, with $c_n=1$ for all $n\ge 0$. Indeed, $\der{\bar{\unity}}\equal\bar{\unity}$ and $\K_{\bar{\unity}}\equal\bar{\unity}\punt\bell\punt\der{\bar{\unity}}^\invumb\equal\bar{\unity}\punt\bell\punt\bar{\unity}^\invumb\equal\singleton$, and then it follows from \eqref{eq:lagrangek2} that 
\begin{equation*}
\big((k+1)\punt\bar{\unity}\big)^{n-k}\simeq
\frac{k+1}{n+1}\big((n+1)\punt\singleton\big)^{n-k}\simeq\frac{k+1}{n+1}(n+1)_{n-k}\quad.
\end{equation*}
Thus, we obtain
\[\prescript{}{\bar{\unity}}{(\bar{\unity},\bar{\unity})}_{n,k}\simeq \frac{(\bar{\unity}+k\punt\bar{\unity})^{n-k}}{(n-k)!}\simeq\frac{\big((1+k)\punt\bar{\unity}\big)^{n-k}}{(n-k)!} \simeq\frac{k+1}{n+1}\frac{(n+1)_{n-k}}{(n-k)!} = \binom{n}{k}.\]

\noindent Now, the g.f. of the umbra $\bar{\unity}$ is $f_{\bar{\unity}}(t)=\frac{1}{1-t}$, so that $\frac{t}{f_{\bar{\unity}}(t)}=t-t^2$ and $\big(\frac{t}{f_{\bar{\unity}}(t)}\big)^\invumb=\frac{1-\sqrt{1-4t}}{2}$. Therefore, we have 
\begin{equation*}
f_{\bar{\unity}}\left(\,\,\left(\frac{t}{f_{\bar{\unity}}(t)}\right)^\invumb\right) = \frac{1}{1-\left(\frac{1-\sqrt{1-4t}}{2}\right)}=\frac{2}{1+\sqrt{1-4t}}.
\end{equation*}
Note that $\frac{2}{1+\sqrt{1-4t}}$ is the generating function of the Catalan numbers. Accordingly, we denote by $C_i^{\sst(m)}$ the generalized $i$-th Catalan number of order $m$ given by 
\begin{equation*}
C_i^{\sst(m)} = \Big[t^i\Big] \left(f_{\bar{\unity}}\left(\,\,\left(\frac{t}{f_{\bar{\unity}}(t)}\right)^\invumb\right)\right)^m = \Big[t^i\Big] \left(\frac{2}{1+\sqrt{1-4t}}\right)^m\,\,,
\end{equation*}
being $C_i^{\sst(1)}$ the classical $i$-th Catalan number whose explicit formula is $C_i=\frac{1}{i+1}\binom{2i}{i}$. By setting $a_i^{\sst(m)} = i!\,C_i^{\sst(m)}$ and substituting the corresponding values in equation \eqref{eq:recdiff-classical}, we get the formula
\begin{equation*}
\binom{n}{k} = \sum_{i=0}^{n-k} C_i^{\sst(m)}\,\binom{n-m-2i}{k-m-i}\,,
\end{equation*}
whenever $2k-n\ge m$.
\end{Example}
\noindent\textbf{Acknowledgments}

\medskip

We would like to thank the anonymous referee for the valuable suggestions and comments.
\bibliographystyle{abbrv}
\bibliography{bibarray2}

\end{document}